\documentclass[a4paper,11pt]{amsart}

\usepackage{epsfig}
\usepackage{amsthm,amsfonts}
\usepackage{amssymb,graphicx,color}
\usepackage[all]{xy}
\usepackage{verbatim}
\usepackage{hyperref}
\usepackage{enumitem}

\newtheorem{theorem}{Theorem}[section]
\newtheorem*{theorem*}{Theorem}

\newtheorem{corollary}[theorem]{Corollary}
\newtheorem{proposition}[theorem]{Proposition}
\newtheorem{definition}[theorem]{Definition}

\newtheorem{example}[theorem]{Example}
\newtheorem{remark}[theorem]{Remark}

\newcommand{\R}{\mathbb{R}}
\newcommand{\C}{\mathbb{C}}

\newcommand{\N}{\mathbb{N}}

\begin{document}

\title[On Zariski's multiplicity problem at infinity]
{On Zariski's multiplicity problem at infinity}

\author{J. Edson Sampaio}
\address[J. Edson Sampaio]{BCAM - Basque Center for Applied Mathematics,
	      Mazarredo, 14 E48009 Bilbao, Basque Country - Spain.   
	      E-mail: {\tt esampaio@bcamath.org} \newline	      
              and       \newline    
              Departamento de Matem\'atica, Universidade Federal do Cear\'a,
	      Rua Campus do Pici, s/n, Bloco 914, Pici, 60440-900, 
	      Fortaleza-CE, Brazil. \newline  
              E-mail: {\tt edsonsampaio@mat.ufc.br}
}

\thanks{The author was partially supported by the ERCEA 615655 NMST Consolidator Grant and also by the Basque Government through the BERC 2014-2017 program and by Spanish Ministry of Economy and Competitiveness MINECO: BCAM Severo Ochoa excellence accreditation SEV-2013-0323.
}
\begin{abstract}
We address a metric version of Zariski's multiplicity conjecture at infinity that says that two complex algebraic affine sets which are bi-Lipschitz homeomorphic at infinity must have the same degree. More specifically, we prove that the degree is a bi-Lipschitz invariant at infinity when the bi-Lipschitz homeomorphism has Lipschitz constants close to 1. In particular, we have that a family of complex algebraic sets bi-Lipschitz equisingular at infinity has constant degree. Moreover, we prove that if two polynomials are weakly rugose equivalent at infinity, then they have the same degree. In particular, we obtain that if two polynomials are rugose equivalent at infinity or bi-Lipschitz contact equivalent at infinity or bi-Lipschitz right-left equivalent at infinity, then they have the same degree.

\end{abstract}

\maketitle

\section{Introduction}

Let $f\colon (\C^n,0)\to (\C,0)$ be the germ of a reduced holomorphic function at the origin and let $(V(f),0)$ be the germ of the zero set of $f$ at origin. In 1971 (see \cite{Zariski:1971}), O. Zariski proposed the following problem:

\begin{enumerate}[leftmargin=0pt]
\item[]{\bf Question A} If $V(f)$ is topologically equivalent to $V(g)$ as germs at the origin $0\in\C^n$, i.e. there exists a homeomorphism $\varphi\colon (\C^n,V(f),0)\to (\C^n,V(g),0)$, then is it true that $m(V(f),0)=m(V(g),0)$?
\end{enumerate}
Although many authors have presented several partial results concerning the question A, it  remains open. In order to know more about Zariski's multiplicity question see, for example, \cite{Eyral:2007}. 

By looking from a metric point of view, and in a more general setting, we have the following metric version of Zariski's multiplicity question (see Chapter 2 in \cite{Chirka:1989}, for a definition of multiplicity of complex analytic sets):
\begin{enumerate}[leftmargin=0pt]
\item[]{\bf Question \~A1($d$)} Let $X\subset \C^n$ and $Y\subset \C^m$ be two complex analytic sets with $\dim X=\dim Y=d$. If their germs at $0\in\C^n$ and $0\in\C^m$, respectively, are bi-Lipschitz homeomorphic, i.e. there exists a bi-Lipschitz homeomorphism $\varphi\colon (X,0)\to (Y,0)$, then is it true that their multiplicities $m(X,0)$ and $m(Y,0)$ are equal?
\end{enumerate}

This question was answered, since when $d\leq 2$, the author jointly with A. Fernandes and J. F. Bobadilla showed that it has a positive answer and when $d\geq 3$ this question has recently been answered negatively by the author in collaboration with L. Birbrair, A. Fernandes and Verbitsky in \cite{BirbrairFSV:2018}. However, let us remark that Question \~A1($d$) was approached in some other works. For instance, G. Comte, in the paper  \cite{Comte:1998}, proved that the multiplicity of complex analytic germs  in $\C^n$ is invariant under bi-Lipschitz homeomorphism with Lipschitz constant close enough to 1. Recently, the author in \cite{Sampaio:2016} (see also \cite{BirbrairFLS:2016}) showed that multiplicity 1 is invariant by bi-Lipschitz homeomorphism and the author jointly with A. Fernandes showed in \cite{FernandesS:2016} that the multiplicity of a complex analytic surface singularity in $\C^3$ is a bi-Lipschitz (embedded) invariant. It was shown also in \cite{FernandesS:2016} that it is enough to address such a question by considering $X$ and $Y$ homogeneous complex algebraic sets. Actually, this result is stated in \cite{FernandesS:2016} for complex analytic hypersurfaces in $\C^n$, however, the proof works for higher codimension complex analytic subsets. Other versions of Question \~A1($d$) were approached by some authors, for example, J.-J. Risler and D. Trotman proved in \cite{RislerT:1997} that if two complex analytic functions are rugose equivalent or bi-Lipschitz right-left equivalent, then they have the same order and G. Comte, P. Milman and D. Trotman showed in \cite{ComteMT:2002} that two complex analytic functions $f,g:(C^n,0)\to (\C,0)$ have the same order, whenever there are positive constants $C$ and $D$ and a homeomorphism $\varphi:(\C^n,0)\to (\C^n,0)$ satisfying 
\begin{enumerate}
\item [{\rm (1)}] $\frac{1}{C}\|z\|\leq \|\varphi(z)\|\leq C\|z\|$, for all $z$ near 0, and
\item [{\rm (2)}] $\frac{1}{D}\|f(z)\|\leq \|g\circ\varphi(z)\|\leq D\|f(z)\|$, for all $z$ near 0.
\end{enumerate}

At this point, we finish this overview on metric versions of the Zariski's multiplicity question and we start to consider the Lipschitz geometry at infinity of complex algebraic sets. 

Let $f\colon \C^n\to \C$ be a reduced polynomial and $X=V(f)$. The degree of the polynomial $f$ is an important integer number associated to $X$; it is called \emph{the degree of $X$}. According to the next example, it is hopeless that degree of $X=V(f)$ comes as a $C^{\infty }$ right invariant. In fact, the degree is not even $C^{\infty }$ right invariant in families. In particular, the degree is not a topological invariant of the embedded subset $X\subset\C^n$.
\begin{example}\label{degree_non}
For each $t\in \C$, let $f_t:\C^2\to \C$ be the polynomial given by $f_t(x,y)=y-tx^2$.  Let $\varphi_t:\C^3\to \C^2$ be the polynomial mapping  given by $\varphi(x,y,t)=(x,y-tx^2)$. Then, $\varphi_t:=\varphi(\cdot,t):\C^2\to \C^2$ is a polynomial automorphism (in particular it is a smooth diffeomorphism) such that $f_t=f_0\circ\varphi_t$, for all $t\in \C$. However, ${\rm deg}(V(f_0))=1$ and ${\rm deg}(V(f_t))=2$, for all $t\not=0$.
\end{example}

In this paper, we deal with the following metric question:

\begin{enumerate}[leftmargin=0pt]
\item[]{\bf Question A1($d$)} Let $X\subset \C^n$ and $Y\subset \C^m$ be two complex algebraic sets with $\dim X=\dim Y=d$. If $X$ and $Y$ are bi-Lipschitz homeomorphic at infinity, in the sense that there exist compact subsets $K_1\subset X$, $K_2\subset Y$ and a bi-Lipschitz homeomorphism $\varphi\colon X\setminus K_1\rightarrow Y\setminus K_2$, then is it true that ${\rm deg}(X)={\rm deg}(Y)$?
\end{enumerate}

The author jointly with A. Fernandes showed in \cite{FernandesS:2017} that degree 1 comes as a bi-Lipschitz invariant at infinity of complex algebraic subsets (see Section \ref{section:preliminaries}, for a definition of degree for higher codimension algebraic sets in $\C^n$). In \cite{BobadillaFS:2017}, the author jointly with J. Fern\'andez de Bobadilla and A. Fernandes showed that the Question  A1($d$) has a positive answer for $d=1$ and $d=2$ and, for each $d\in \N$, A1($d$) and \~A1($d$) are equivalent questions. Since ${\rm deg}(X)=m(X,0)$, when $X$ is a homogeneous complex algebraic set, the sets presented in (\cite{BirbrairFSV:2018}, Theorem 4.3) show that Question A1($d$) has, in general, a negative answer when $d>2$. Thus, the aim of this paper is to present some cases that Question A1($d$) has a positive answer.

Let us describe how this paper is organized. Section \ref{section:preliminaries} is dedicated to present the notions of tangent cones at infinity, degree and relative multiplicities at infinity of complex algebraic subsets in $\C^n$ and, also, bi-Lipschitz homeomorphisms at infinity of such subsets. Section \ref{section:mainresults} is dedicated to proving the main results of the paper, we prove that the degree of a complex algebraic set is invariant under bi-Lipschitz homeomorphism with Lipschitz constant close enough to 1. In particular, in contrast with the example \ref{degree_non}, we obtain that the degree is constant in a family which is bi-Lipschitz equisingular at infinity.
Moreover, we prove that if two polynomials are weakly rugose equivalent at infinity, then they have the same degree. In particular, we obtain that two polynomials have the same degree, if they are rugose equivalent at infinity or bi-Lipschitz contact equivalent at infinity or bi-Lipschitz right-left equivalent at infinity.

\bigskip

\noindent{\bf Acknowledgements}. The author would like to thank Alexandre Fernandes for his interest on this work as well as for his suggestions in the organization of this article. The author would like to thank anonymous referee for corrections and suggestions in writing this article.

\section{Preliminaries}\label{section:preliminaries}
\subsection{Tangent cones}
In this subsection, we set the exact notion of tangent cone that we will use throughout the paper and we list some of its properties.

\begin{definition}
Let $A\subset \R^n$ be an unbounded subset. We say that $v\in \R^n$ is a {\bf tangent vector of} $A$ {\bf at infinity} if there is a sequence of points $\{x_i\}_{i\in \N}\subset A$ such that $\lim\limits_{i\to \infty} \|x_i\|=+\infty $ and there is a sequence of positive numbers $\{t_i\}_{i\in \N}\subset\R^+$ such that 
$$\lim\limits_{i\to \infty} \frac{1}{t_i}x_i= v.$$
Let $C_{\infty }(A)$ denote the set of all tangent vectors of $A$ at infinity. This subset $C_{\infty }(A)\subset\R^n$ is called {\bf the tangent cone of} $A$ {\bf at infinity}.
\end{definition}

\begin{proposition}[Proposition 4.4 in \cite{FernandesS:2017}]\label{curve_infinity}
Let $Z\subset \R^n$ be an unbounded semialgebraic set. A vector $v\in\R^n$ belongs to $C_{\infty }(Z)$ if, and only if, there exists a continuous semialgebraic  curve $\gamma\colon (\varepsilon ,+\infty )\to Z$  such that $\lim\limits _{t\to +\infty }|\gamma(t)|=+\infty $ and $\gamma(t)=tv+o_{\infty }(t),$ where $g(t)=o_{\infty }(t)$ means $\lim\limits _{t\to +\infty }\frac{g(t)}{t}=0$.
\end{proposition}

Let $X\subset\C^n$ be a complex algebraic subset. Let $\mathcal{I}(X)$ be the ideal of $\C[x_1,\cdots,x_n]$ given by the polynomials which vanish on $X$. For each $f\in\C[x_1,\cdots,x_n]$, let us denote by $f^*$ the homogeneous polynomial composed of the monomials in $f$ of maximum degree.

\begin{proposition}[Theorem 1.1 in \cite{LeP:2016}]\label{algebricity}
Let $X\subset\C^n$ be a complex algebraic subset. Then, $C_{\infty }(X)$ is the affine algebraic set $V(\langle f^*;\, f\in \mathcal{I}(X)\rangle) $.
\end{proposition}
Among other things, this result above says that tangent cones at infinity of complex algebraic sets in $\C^n$ are complex algebraic subsets as well.

\subsection{Degree and relative multiplicities at infinity}\label{section:lelong} 
This Subsection is closely related to Subsection 1.4 in \cite{BobadillaFS:2017}.

 Let $X\subset \C^n$ be a complex algebraic set with $p=\dim X\geq 1$ and let $X_1,\cdots,X_r$ be the irreducible components of $C_{\infty }(X)$. Below we present a definition of degree which suits better our purposes; for more details about degree see \cite{Chirka:1989}.

Let $\pi\colon\C^n\to \C^p$ be a linear projection such that 
$$\pi^{-1}(0)\cap(C_{\infty }(X))=\{0\}.$$ 
Therefore, $\pi|_{ X}\colon X\rightarrow \C^p$ is a ramified cover with degree equal to $k$. It is well known that the number $k$ does not depend on $\pi$ and, this number is called the {\bf degree of $X$} and denoted by ${\rm deg} (X)$ (see \cite{Chirka:1989}, Corollary 1 on page 126). In particular, $\pi|_{X_j}\colon X_j\rightarrow \C^p$ (resp. $\pi|_{ C_{\infty }(X)}\colon C_{\infty }(X)\rightarrow \C^p$) is a ramified cover with degree equal to ${\rm deg} (X_j)$ (resp. ${\rm deg} (C_{\infty }(X))$), for each $j=1,\cdots,r$. 

\begin{remark}
Let $f\colon \C^n\to \C$ be a reduced polynomial and $X=V(f)$. Then, ${\rm deg}(X)={\rm deg}(f)$.
\end{remark}
Moreover, if the ramification locus of $\pi|_X$ (resp. $\pi|_{ C_{\infty }(X)}$) is not empty, it is a codimension $1$ complex algebraic 
subset $\sigma(X)$ (resp. $\sigma(C_{\infty }(X))$) of $\C^p$. Let us denote 
$\Sigma=\pi|_X^{-1}(\sigma(X))$ and $\Sigma'=\pi|_{C_{\infty }(X)}^{-1}(\sigma(C_{\infty }(X)))$.

Fix $j\in \{1,\cdots,r\}$. For a point $v\in X_j\setminus (C_{\infty }(\Sigma)\cup C_{\infty }(\Sigma'))$, let $\eta, R >0$ such that 
$$C_{\eta,R }(v')\!:=\!\{w\in \C^p|\, \exists t>0; \|tv'-w\|\leq \eta t \}\setminus B_R(0)\subset \C^p\setminus \sigma(X)\cup \sigma(C_{\infty }(X)),$$
where $v'=\pi(v)$. Thus, by using the definition of degree and since $C_{\eta,R }(v')$ (resp. $C_{\eta,R }(v')$) is connected, we have that the number of connected components of $\pi|_X^{-1}(C_{\eta,R }(v'))$ (resp. $\pi|_{ X_j}^{-1}(C_{\eta,R }(v'))$) is equal to ${\rm deg} (X)$ (resp. ${\rm deg} (X_j)$). Moreover, there exist a connected component $V$ of $\pi|_{ X_j}^{-1}(C_{\eta,R }(v'))$ such that $v\in V$ and a compact subset $K\subset \C^n$ such that for each connected component $A_i$ of $\pi|_X^{-1}(C_{\eta,R }(v'))$, we have $C_{\infty }(A_i)\cap (\C^n\setminus K)\subset \pi|_{ C_{\infty }(X)}^{-1}(C_{\eta,R }(v'))$. Then, we denote by $k_X^{\infty }(v)$ the number of connected components $A_i$ such that $C_{\infty }(A_i)\cap (\C^n\setminus K)\subset V$. By definition, we can see that $k_X^{\infty }$ is locally constant and as $X_j\setminus (C_{\infty }(\Sigma)\cup C_{\infty }(\Sigma'))$ is connected, $k_X^{\infty }$ is constant on $X_j\setminus (C_{\infty }(\Sigma)\cup C_{\infty }(\Sigma'))$. Thus, we define $k_X^{\infty }(X_j)=k_X^{\infty }(v)$. In particular, $k_X^{\infty }(w)=k_X^{\infty }(v)$ for all $w\in \pi^{-1}(v')\cap X_j$. Therefore, by using the definition of degree of  $X$ once more, we obtain the following formula
\begin{equation}\label{kurdyka-raby}
{\rm deg}(X)=\sum\limits_{j=0}^r k_X^{\infty }(X_j)\cdot{\rm deg}(X_j).
\end{equation}

The numbers $k_X^{\infty }(X_1),\cdots, k_X^{\infty }(X_r)$ are called {\bf relative multiplicities at infinity of $X$}.

\begin{definition}
Let $X\subset \R^n$ and $Y\subset\R^m$ be two subsets. We say that $X$ and $Y$ are {\bf bi-Lipschitz homeomorphic at infinity}, if there exist compact subsets $K\subset\R^n$ and $\widetilde K\subset \R^m$ and a bi-Lipschitz homeomorphism $\phi \colon X\setminus K\rightarrow Y\setminus \widetilde K$.
\end{definition}

We finish this Section by recalling the invariance of the relative multiplicities at infinity under bi-Lipschitz homeomorphisms at infinity.


\begin{proposition}[Theorem 3.1 in \cite{BobadillaFS:2017}]\label{multiplicities}
Let $X\subset\C^n$ and $Y\subset\C^m$ be complex algebraic subsets, with pure dimension $p=\dim X=\dim Y$, and let $X_1,\dots,X_r$ and $Y_1,\dots,Y_s$ be the irreducible components of the tangent cones at infinity $C_{\infty }(X)$ and $C_{\infty }(Y)$ respectively. If $X$ and $Y$ are bi-Lipschitz homeomorphic at infinity, then $r=s$ and, up to a re-ordering of indices,  $k_X^{\infty }(X_j)=k_Y^{\infty }(Y_j)$, $\forall \ j$.
\end{proposition}

\section{Degree as a bi-Lipschitz Invariant at Infinity}\label{section:mainresults}
\subsection{Degree of complex algebraic sets}
The next result is the analogue at infinity of Comte's result mentioned in the introduction.
\begin{theorem}\label{main-theorem}
Let $X\subset \C^n$ and $Y\subset \C^m$ be two complex algebraic sets with $\dim X=\dim Y=d$ and $M=\max \{{\rm deg}(X),{\rm deg}(Y)\}$. If there are compact subsets $K\subset \C^n$ and $\widetilde K\subset \C^m$, constants $C_1,C_2>0$ and a bi-Lipschitz homeomorphism $\varphi:X\setminus K\to Y\setminus \widetilde K$ such that 
$$
\frac{1}{C_1}\|x-y\|\leq \|\varphi(x)-\varphi(y)\|\leq C_2\|x-y\|, \quad \forall x,y \in X\setminus K
$$
and $(C_1C_2)^{2d}\leq 1+\frac{1}{M}$, then ${\rm deg}(X)={\rm deg}(Y).$
\end{theorem}
\begin{proof}
Let $X_1,\dots,X_r$ and $Y_1,\dots,Y_s$ be the irreducible components of the tangent cones at infinity $C_{\infty }(X)$ and $C_{\infty }(Y)$ respectively. Considering $X$ and $Y$, respectively, as the sets $X\times\{0\}$ and $\{0\}\times Y$ in $\C^{n+m}=\C^n\times \C^m$, we have by the proof of Lemma 3.1 in \cite{Sampaio:2016}, that there are $C>0$ and a bi-Lipschitz homeomorphism $\Phi:\C^{n+m}\to \C^{n+m}$ such that $\Phi|_{X\setminus K}=\varphi$ and
$$
\frac{1}{C}\|x-y\|\leq \|\Phi(x)-\Phi(y)\|\leq C\|x-y\|, \quad \forall x,y \in X\setminus K.
$$
Thus, by the proof of the Theorem 4.5 in \cite{FernandesS:2017}, there is a bi-Lipschitz homeomorphism $d\varphi:\C^{n+m}\to \C^{n+m}$ such that $d\varphi(0)=0$, $d\varphi(C_{\infty }(X))= C_{\infty }(Y)$ and
$$
\frac{1}{C}\|v-w\|\leq \|d\varphi(v)-d\varphi(w)\|\leq C\|v-w\|, \quad \forall v,w \in C_{\infty }(X).
$$
Moreover, there is a sequence $\{t_j\}\subset \N$ such that $\varphi_{t_j}\rightarrow d\varphi$ uniformly on compact subsets of $\C^{n+m}$, where each mapping $\varphi_k:\C^{n+m}\to \C^{n+m}$ is given by $\varphi_k(v)=\frac{1}{k}\Phi(kv)$ for all $v\in\C^{n+m}$.

\noindent{\bf Claim.} $\frac{1}{C_1}\|v-w\|\leq \|d\varphi(v)-d\varphi(w)\|\leq C_2\|v-w\|, \quad \forall v,w \in C_{\infty }(X).$

Let $v\in C_{\infty }(X)$. By Proposition \ref{curve_infinity}, there is a proper curve $\gamma\colon (\varepsilon ,+\infty )\to X$  such that $\lim\limits _{t\to +\infty }|\gamma(t)|=+\infty $ and $\gamma(t)=tv+o_{\infty }(t)$. Then, we obtain 
$$
\textstyle{\left\|\frac{\Phi(t_jv)}{t_j}-\frac{\Phi(\gamma(t_j))}{t_j}\right\|}=\frac{o_{\infty }(t_j)}{t_j}\to 0 \mbox{ as } j\to +\infty .
$$
Therefore, 
$$\textstyle{\lim\limits _{j\to +\infty}\frac{\Phi(t_jv)}{t_j}=\lim\limits _{j\to +\infty}\frac{\Phi(\gamma(t_j))}{t_j}=d\varphi(v).}$$ 
As $\Phi|_{X\setminus K}=\varphi$, we have 
\begin{equation} \label{eq_lim}
\textstyle{\lim\limits _{j\to +\infty}\frac{\varphi(\gamma(t_j))}{t_j}= d\varphi(v).}
\end{equation} 
Therefore, if $v,w\in C_{\infty }(X)$, there are curves $\gamma, \beta\colon (\varepsilon ,+\infty )\to X$  such that  $\gamma(t)=tv+o_{\infty }(t)$ and $\beta(t)=tw+o_{\infty }(t)$. Thus, by the hypothesis of the theorem, we get
$$
\textstyle{\frac{1}{C_1}\left\|\frac{\gamma(t_j)}{t_j}-\frac{\beta(t_j)}{t_j}\right\|\leq \left\|\frac{\varphi(\gamma(t_j))}{t_j}-\frac{\varphi(\beta(t_j))}{t_j}\right\|\leq C_2\left\|\frac{\gamma(t_j)}{t_j}-\frac{\beta(t_j)}{t_j}\right\|}.
$$
Passing to the limit $j\to +\infty $ and using (\ref{eq_lim}), we obtain
$$
\frac{1}{C_1}\|v-w\|\leq \|d\varphi(v)-d\varphi(w)\|\leq C_2\|v-w\|.
$$

By Proposition \ref{multiplicities}, $r=s$ and, up to a re-ordering of indices, $k_X^{\infty }(X_j)=k_Y^{\infty }(Y_j)$ and $Y_j=d\varphi(X_j)$, $\forall \ j$. Moreover, by eq. (\ref{kurdyka-raby}), we get
$${\rm deg}(X)=\sum\limits_{j=0}^r k_X^{\infty }(X_j)\cdot{\rm deg}(X_j)$$
and 
$${\rm deg}(Y)=\sum\limits_{j=0}^r k_Y^{\infty }(Y_j)\cdot{\rm deg}(Y_j).$$
In particular, for each $j$, $M_j=\max \{{\rm deg}(X_j),{\rm deg}(Y_j)\}\leq M$. Since $X_j$ and $Y_j$ are homogeneous algebraic sets, we have ${\rm deg}(X_j)=m(X_j,0)$ and ${\rm deg}(Y_j)=m(Y_j,0)$. By Theorem 1 in \cite{Comte:1998}, ${\rm deg}(X_j)={\rm deg}(Y_j)$, for all $j$. Therefore,
${\rm deg}(X)={\rm deg}(Y).$
\end{proof}

\noindent{\bf Notation. }Let $A\subset \R^m$, $B\subset \R^k$ and $f:A\to B$ be a Lipschitz function. We define the Lipschitz constant of $f$ by
$$
\textstyle{Lip(f):=\sup\left\{\frac{\|f(x)-f(y)\|}{\|x-y\|};x,y\in A\mbox{ and } x\not=y\right\}.}
$$

\begin{definition}
The family of complex algebraic sets $\{X_t\}_{t\in [0,1]}$ in $\C^n$ is said to be {\bf bi-Lipschitz equi\-sin\-gular at infinity}, if there are a compact subset $K\subset \C^n$ and a mapping $\varphi:(X_0\setminus K)\times[0,1] \to \C^n$ such that 
\begin{itemize}
\item [(i)] for each $t\in[0,1]$, $\varphi((X_0\setminus K)\times \{t\})=X_t\setminus K_t$ for some compact $K_t\subset \C^n$ and $\varphi_t:=\varphi(\cdot,t):X_0\setminus K \to X_t\setminus K_t$ is a bi-Lipschitz homeomorphism with $\varphi_0={\rm id}$ and
\item [(ii)] $\lim\limits _{t\to 0^+}Lip(\varphi_t)=\lim\limits_ {t\to 0^+}Lip(\varphi_t^{-1})=1$.
\end{itemize}
In this case, we say that $\varphi$ is a {\bf bi-Lipschitz deformation of $X_0$ at infinity}.
\end{definition}

\begin{theorem}\label{lip_trivial}
Let $\{X_t\}_{t\in [0,1]}$ be a family of complex algebraic sets. If $\{X_t\}_{t\in [0,1]}$ is bi-Lipschitz equisingular at infinity, then there is $\delta\in (0,1]$ such that ${\rm deg}(X_t)={\rm deg}(X_0)$, for all $t\in [0,\delta]$.
\end{theorem}
\begin{proof}
Let $\varphi:(X_0\setminus K)\times [0,1] \to \C^n$ be a bi-Lipschitz deformation of $X_0$ at infinity. Thus, $\varphi_t:=\varphi(\cdot,t):(X_0\setminus K) \to X_t\setminus K_t$ is a bi-Lipschitz homeomorphism and $\lim\limits _{t\to 0} C_t=\lim\limits _{t\to 0} C_t'=1$, where $C_t$ and $C_t'$ are, respectively, the Lipschitz constants of the mappings $\varphi_t$ and $\varphi_t^{-1}$. As it was done in the proof of the Theorem \ref{main-theorem}, for each $t\in [0,1]$ there is $\psi_t:C_{\infty }(X_0)\to C_{\infty }(X_t)$ such that
$$
\frac{1}{C_t'}\|v-w\|\leq \|\psi_t(v)-\psi_t(w)\|\leq C_t\|v-w\|, \quad \forall v,w\in C_{\infty }(X_0).
$$
Thus, if $Y_{0,1},...,Y_{0,r}$ are the irreducible components of $C_{\infty }(X_0)$, then by Lemma A.8 in \cite{Gau-Lipman:1983}, for each $i=1,...,r$, there is an irreducible component $Y_{t,i}$ of $C_{\infty }(X_t)$ such that $\psi_t(Y_{0,i})=Y_{t,i}$, since $\psi_t$ is, in particular, a homeomorphism. By Theorem 2 in \cite{Comte:1998}, there is $t_i\in (0,1]$ such that ${\rm deg}(Y_{t,i},0)={\rm deg}(Y_{0,i},0)$ for all $t\in [0,t_i]$, since $Y_{0,i}$ and $Y_{t,i}$ are homogeneous complex algebraic sets. Using that the relative multiplicities at infinity are bi-Lipschitz invariant at infinity, we obtain ${\rm deg}(X_t,0)={\rm deg}(X_0,0)$ for all $t\in [0,\delta]$, where $\delta=\min \{t_1,...,t_r\}$.
\end{proof}
\begin{remark}
{\rm The Theorem \ref{lip_trivial} above is still true even if the family $\{\varphi_t\}$ of bi-Lipschitz homeomorphisms does not satisfy $\varphi_0={\rm id}$.}
\end{remark}

\subsection{Degree of polynomials}
\begin{definition}
We say that two polynomials $f,g:\C^n\to \C$ are {\bf bi-Lipschitz contact equivalent at infinity}, if there are compact subsets $K,\widetilde K\subset \C^n$, a constant $C>0$ and a bi-Lipschitz homeomorphism $\varphi:\C^n\setminus K\to \C^n\setminus \widetilde K$  such that 
$$\frac{1}{C}\|f(x)\|\leq \|g\circ\varphi(x)\|\leq C\|f(x)\|, \quad \forall x\in \C^n\setminus K.$$
\end{definition}

\begin{definition}
We say that two polynomials $f,g:\C^n\to \C^m$ are {\bf rugose equivalent at infinity}, if there are compact subsets $K,\widetilde K\subset \C^n$, constants $C_1,C_2>0$ and a bijection $\varphi:\C^n\setminus K\to \C^n\setminus \widetilde K$  such that 
\begin{enumerate}
\item [{\rm (1)}] $\frac{1}{C_1}\|x-y\|\leq \|\varphi(x)-\varphi(y)\|\leq C_1\|x-y\|$, for all $x\in \C^n\setminus K$ and $y\in f^{-1}(0)\setminus K$;
\item [{\rm (2)}] $\frac{1}{C_2}\|f(x)\|\leq \|g\circ\varphi(x)\|\leq C_2\|f(x)\|, \quad \forall x\in \C^n\setminus K.$
\end{enumerate}
\end{definition}


The next result is a consequence of Theorem 3.7 in (\cite{FernandesS:2017b}, Theorem 3.7). However, here we present a direct proof without using the global \L ojasiewicz inequality proved in \cite{JiKS:1992}.
\begin{theorem}
Let $f,g:\C^n\to \C$ be two polynomials. If $f$ and $g$ are bi-Lipschitz contact equivalent at infinity, then ${\rm deg}(f)={\rm deg}(g)$.
\end{theorem}
\begin{proof}
Let us denote $X=\{x\in \C^n;\, f(x)=0\}$ and $Y=\{x\in \C^n;\, g(x)=0\}$. We have that $X$ and $Y$ are bi-Lipschitz homeomorphic at infinity. By Theorem 4.5 in \cite{FernandesS:2017} and Proposition \ref{algebricity}, $C_{\infty }(X)$ and $C_{\infty }(Y)$ are closed and bi-Lipschitz homeomorphic sets. By hypothesis, there are compact subsets $K,\widetilde K\subset \C^n$, positive constants $C_1$ and $C_2$ and $\varphi:\C^n\setminus K\to \C^n\setminus \widetilde K$  such that 
$$
\frac{1}{C_1}\|x-y\|\leq \|\varphi(x)-\varphi(y)\|\leq C_1\|x-y\|,\, \forall x,y\in \C^n\setminus K
$$
and
$$
\frac{1}{C_2}\|f(x)\|\leq \|g\circ\varphi(x)\|\leq C_2\|f(x)\|, \quad \forall x\in \C^n\setminus K.
$$
Let us suppose that ${\rm deg}(f)<{\rm deg}(g)=k$. Let $S=\{n_j\}_{j\in\N}\subset \N$ be a sequence such that
$$n_j\to +\infty  \quad\mbox{and} \quad \frac{\varphi(n_jv)}{n_j}\to d\varphi(v),$$
as in Theorem 4.5 in \cite{FernandesS:2017}. Moreover, $d\varphi: \C^n\to \C^n$ is a bi-Lipschitz homeomorphism. Then, there is $v\in \C^n$ such that $d\varphi(v)\in \C^n\setminus \{x\in \C^n;\, g^*(x)=0\}$, where $g^*$ is the homogeneous polynomial composed of the monomials in $g$ of maximum degree. Therefore, 
$$
\frac{\|g\circ\varphi(n_jv)\|}{n_j^k}\leq C_2\frac{\|f(n_jv)\|}{n_j^k} , \quad \forall n_j\in S.
$$
By taking $j\to +\infty $, we obtain $\|g^*(d\varphi(v))\|\leq 0$, which is a contradiction. Then, ${\rm deg}(f)\geq {\rm deg}(g)=k$ and by using $\varphi^{-1}$ instead of $\varphi$, we obtain the other inequality. Therefore, ${\rm deg}(g)={\rm deg}(f)$.
\end{proof}

\begin{definition}\label{def_weakly_rugose}
We say that two polynomial mappings $F,G:\C^n\to \C^m$ are {\bf weakly rugose equivalent at infinity}, if there are compact subsets $K,\widetilde K\subset \C^n$, constants $C_1,C_2>0$ and a bijection $\varphi:\C^n\setminus K\to \C^n\setminus \widetilde K$  such that \begin{enumerate}
\item [{\rm (1)}] there exist $y_0\in \C^n\setminus K$ and $w_0\in \C^n\setminus \widetilde K$ such that $\|\varphi(x)-\varphi(y_0)\|\leq C_1\|x-y_0\|$, for all $x\in \C^n\setminus K$ and $\|\varphi^{-1}(z)-\varphi^{-1}(w_0)\|\leq C_1\|z-w_0\|$, for all $z \in \C^n\setminus \widetilde K$;
\item [{\rm (2)}] $\frac{1}{C_2}\|F(x)\|\leq \|G\circ\varphi(x)\|\leq C_2\|F(x)\|, \quad \forall x\in \C^n\setminus K.$
\end{enumerate}
\end{definition}
Let $f:\C^n\to \C$ be a polynomial. Then, for each $r>0$, we define
$$
\delta_{r,\infty }(f)=\textstyle{\inf \{\delta ;\,\frac{|f(z)|}{\|z\|^{\delta}} \mbox{ is bounded on } \C^n\setminus B_r(0)\}.}
$$
Remark that $\delta_{r,\infty }(f)$ does not depend on $r>0$. Thus, we define this common number by $\delta_{\infty }(f)$.
\begin{proposition}\label{expoent}
Let $f:\C^n\to \C$ be a polynomial. Then, ${\rm deg}(f)=\delta_{\infty }(f)$.
\end{proposition}
\begin{proof}
If $\delta<d={\rm deg}(f)$ and $f=f_0+f_1+...+f_d$, then we choose $v\not\in V(f_d)$. Thus, $\lim \limits _{t\to +\infty }\frac{|f(tv)|}{t^{\delta}}=+\infty $. Then, $\delta_{\infty }(f)\geq {\rm deg}(f)$.

If $\delta>{\rm deg}(f)$, then $\lim \limits _{\|z\|\to +\infty }\frac{|f(z)|}{\|z\|^{\delta}}=0.$ Thus, there exists $r>0$ such that $\frac{|f(z)|}{\|z\|^{\delta}}\leq 1$, for all $z\not\in  \C^n\setminus B_r(0)$. This implies $\delta_{\infty }(f)\leq {\rm deg}(f)$. Therefore, $\delta_{\infty }(f)= {\rm deg}(f)$.
\end{proof}
The next result is an analogue at infinity of the result of Comte, Milman and Trotman that was mentioned in the introduction.
\begin{theorem}\label{weakly_rugose}
Let $f,g:\C^n\to \C$ be two polynomials. If $f$ and $g$ are weakly rugose equivalent at infinity, then ${\rm deg}(f)={\rm deg}(g)$.
\end{theorem}
\begin{proof}
By hypothesis, there are compact subsets $K,\widetilde K\subset \C^n$, constants $C_1,C_2>0$ and a bijection $\varphi:\C^n\setminus K\to \C^n\setminus \widetilde K$  such that 
\begin{enumerate}
\item [{\rm (1)}] there exist $y_0\in \C^n\setminus K$ and $w_0\in \C^n\setminus \widetilde K$ such that $\|\varphi(x)-\varphi(y_0)\|\leq C_1\|x-y_0\|$, for all $x\in \C^n\setminus K$ and $\|\varphi^{-1}(z)-\varphi^{-1}(w_0)\|\leq C_1\|z-w_0\|$, for all $z \in \C^n\setminus \widetilde K$;
\item [{\rm (2)}] $\frac{1}{C_2}\|f(x)\|\leq \|g\circ\varphi(x)\|\leq C_2\|f(x)\|, \quad \forall x\in \C^n\setminus K.$
\end{enumerate}

Let $r>0$ be a positive number satisfying 
 $\widetilde{r}=C_1^{-1}(r-\|\varphi^{-1}(w_0)\|)-\|w_0\|>0$ and $K \subset B_{r}(0)$. Thus, $\|x\|>r$ implies $\|\varphi(x)\|> \widetilde{r}$. In fact, if $\|x\|>r$, then
$$
 \begin{array}{lllll}
\|x-\varphi^{-1}(w_0)\| & \geq &  \|x\|-\|\varphi^{-1}(w_0)\| & > & r-\|\varphi^{-1}(w_0)\|
\end{array}
$$
and, by hypothesis, we have 
$$\|x-\varphi^{-1}(w_0)\|=\|\varphi^{-1}(\varphi(x))-\varphi^{-1}(w_0)\|\leq C_1\|\varphi(x)-w_0\|,$$ 
then
$$
 \begin{array}{lllll}		
C_1(\|\varphi(x)\|+\|w_0\| )    & \geq & C_1(\|\varphi(x)-w_0\|) & > & r-\|\varphi^{-1}(w_0)\|,
\end{array}
$$		                   
and, therefore, $\|\varphi(x)\| \geq  C_1^{-1}(r-\|\varphi^{-1}(w_0)\|)-\|w_0\|=\widetilde{r}$.

Moreover, we have the following
\begin{eqnarray*}
\frac{|f(x)|}{\|x\|^{\delta}}& = & \frac{|f(x)|}{\|\varphi(x)\|^{\delta}}\frac{\|\varphi(x)\|^{\delta}}{\|x\|^{\delta}}\\
							 & \leq & C_2\frac{|g(\varphi(x))|}{\|\varphi(x)\|^{\delta}}\left(\frac{\|\varphi(x)-\varphi(y_0)\|+\|\varphi(y_0)\|}{\|x\|}\right)^{\delta}\\
							 & \leq & C_2\frac{|g(\varphi(x))|}{\|\varphi(x)\|^{\delta}}\left(C_1\frac{\|x-y_0\|}{\|x\|}+\frac{\|\varphi(y_0)\|}{\|x\|}\right)^{\delta}\\
							 & \leq & C_2\frac{|g(\varphi(x))|}{\|\varphi(x)\|^{\delta}}\left(C_1+C_1\frac{\|y_0\|}{r}+\frac{\|\varphi(y_0)\|}{r}\right)^{\delta}\\
							 & = & C\frac{|g(\varphi(x))|}{\|\varphi(x)\|^{\delta}},
\end{eqnarray*}
for all $x\in \C^n\setminus B_{r}(0)$. Thus, if $\frac{|g(z)|}{\|z\|^{\delta}}$ is bounded on $\C^n\setminus B_{\widetilde{r}}(0)$, then $\frac{|f(x)|}{\|x\|^{\delta}}$ is bounded on $\C^n\setminus B_{r}(0)$. This implies 
$$
\textstyle{\{\rho ;\,\frac{|g(z)|}{\|z\|^{\rho}} \mbox{ is bounded on } \C^n\setminus B_{\widetilde{r}}(0)\}\subset \{s ;\,\frac{|f(x)|}{\|x\|^{s}} \mbox{ is bounded on } \C^n\setminus B_{r}(0)\},}
$$
Then, we obtain $\delta_{r,\infty }(f)\leq \delta_{\widetilde{r},\infty }(g)$ and, since $\delta_{r,\infty }(f)=\delta_{\infty }(f)$ and $\delta_{\widetilde{r},\infty }(g)=\delta_{\infty }(g)$, we have $\delta_{\infty }(f)\leq \delta_{\infty }(g)$. Therefore, by Proposition \ref{expoent}, ${\rm deg}(f)\leq{\rm deg}(g)$. Similarly, we obtain ${\rm deg}(g)\leq{\rm deg}(f)$. Thus, we have the equality ${\rm deg}(g)={\rm deg}(f)$.
\end{proof}

\begin{definition}
We say that two polynomials $f,g:\C^n\to \C$ are {\bf bi-Lipschitz right-left equivalent at infinity}, if there are compact subsets $K,\widetilde K\subset \C^n$, a constant $C>0$ and bi-Lipschitz homeomorphisms $\varphi:\C^n\setminus K\to \C^n\setminus \widetilde K$ and $\phi:\C\to \C$ such that $f(x)=\phi\circ g\circ\varphi(x),$ $\forall x\in \C^n\setminus K.$
\end{definition}
The following result follows directly from the definitions.
\begin{proposition}\label{implica}
Let $f,g:\C^n\to \C$ be two polynomials. Let us consider the following statements:
\begin{enumerate}
\item [{\rm (1)}] $f$ and $g$ are bi-Lipschitz right-left equivalent at infinity;
\item [{\rm (2)}] $f$ and $g$ are bi-Lipschitz contact equivalent at infinity;
\item [{\rm (3)}] $f$ and $g$ are rugose equivalent at infinity;
\item [{\rm (4)}] $f$ and $g$ are weakly rugose equivalent at infinity.
\end{enumerate}
Then, $(1)\Rightarrow (2)\Rightarrow (3)\Rightarrow (4)$.
\end{proposition}
We finish this paper by stating some direct consequences of Theorem \ref{weakly_rugose} and Proposition \ref{implica}.
\begin{corollary}[\cite{FernandesS:2017b}, Theorem 3.7]
Let $f,g:\C^n\to \C$ be two polynomials. If $f$ and $g$ are rugose equivalent at infinity, then ${\rm deg}(f)={\rm deg}(g)$.
\end{corollary}
\begin{corollary}
Let $f,g:\C^n\to \C$ be two polynomials. If $f$ and $g$ are bi-Lipschitz contact equivalent at infinity, then ${\rm deg}(f)={\rm deg}(g)$.
\end{corollary}
\begin{corollary}
Let $f,g:\C^n\to \C$ be two polynomials. If $f$ and $g$ are bi-Lipschitz right-left equivalent at infinity, then ${\rm deg}(f)={\rm deg}(g)$.
\end{corollary}


\begin{thebibliography}{99}

\bibitem{BirbrairFLS:2016}
{Birbrair, L.; Fernandes, A.; L\^E D. T. and Sampaio, J. E.}
{\em Lipschitz regular complex algebraic sets are smooth}.
Proceedings of the American Mathematical Society 144 (2016), no. 3, 983--987.

\bibitem{BirbrairFSV:2018}
{Birbrair, Lev; Fernandes, Alexandre; Sampaio, J. Edson and Verbitsky, Misha}.
{\em Multiplicity of singularities is not a bi-Lipschitz invariant}.
arXiv:1801.06849v1 [math.AG], 2018.

\bibitem{Chirka:1989}
{Chirka, E. M.}
{\em Complex analytic sets}.
Kluwer Academic Publishers, 1989.

\bibitem{Comte:1998}
{Comte, Georges}.
{\em Multiplicity of complex analytic sets and bi-Lipschitz maps}.
Real analytic and algebraic singularities (Nagoya/Sapporo/Hachioji, 1996) Pitman Res. Notes Math. Ser. 381 (1998),  182--188.

\bibitem{ComteMT:2002}
{Comte, Georges; Milman, Pierre and Trotman, David}
{\em On Zariski's multiplicity problem}.
Proceedings of the American Mathematical Society 130 (2002), no. 7, 2045--2048.


\bibitem{Eyral:2007} 
{Eyral, C.}
{\em Zariski's multiplicity questions - A survey}.
New Zealand Journal of Mathematics 36 (2007), 253--276.

\bibitem{FernandesS:2016}
{Fernandes, Alexandre and Sampaio, J. Edson}.
{\em Multiplicity of analytic hypersurface singularities under bi-Lipschitz homeomorphisms}.
Journal of Topology 9 (2016), 927--933.

\bibitem{FernandesS:2017}
{Fernandes, Alexandre and Sampaio, J. Edson}.
{\em On Lipschitz rigidity of complex analytic sets}.
arXiv:1705.03085v3 [math.AG], 2018.

\bibitem{FernandesS:2017b}
{Fernandes, Alexandre and Sampaio, J. Edson}.
{\em Degree of complex algebraic sets under bi-Lipschitz homeomorphisms at infinity}.
arXiv:1706.06614v1 [math.AG], 2017.

\bibitem{BobadillaFS:2017}
{Fern\'andez de Bobadilla, Javier; Fernandes, Alexandre and Sampaio, J. Edson}.
{\em Multiplicity and degree as bi-Lipschitz invariants for complex sets}.
J. of Topology, 11 (2018), no. 4, 957-965.

\bibitem{Gau-Lipman:1983}
{Gau, Y.-N. and Lipman, J.}
{\em Differential invariance of multiplicity on analytic varieties}.
Inventiones Mathematicae 73 (1983), no. 2, 165--188.


\bibitem{JiKS:1992} 
{Ji, Shanyu; Kollar, Janos and Shiffiman, Bernard.}
{\em A global \L ojasiewicz inequality for algebraic varieties}. 
Transactions of the American Mathematical Society 329 (1992), no. 2, 813--818


\bibitem{LeP:2016}
{ L\^e, C\^ong-Tr\`inh and Pham, Tien-Son}.
{\it On tangent cones at infinity of algebraic varieties}.
Journal of Algebra and Its Applications, vol. 16 (2),  1850143 (10 pages) (2018).



\bibitem{RislerT:1997}
{Risler, Jean-Jacques and Trotman, David}
{\em Bi-Lipschitz invariance of the multiplicity}.
Bull. London Math. Soc. 29 (1997), 200--204.


\bibitem{Sampaio:2016}
{Sampaio, J. Edson}
{\em Bi-Lipschitz homeomorphic subanalytic sets have bi-Lipschitz homeomorphic tangent cones}.
Selecta Math. (N.S.) 22 (2016), no. 2, 553--559, .




\bibitem{Zariski:1971} 
{Zariski, O.}
{\em Some open questions in the theory of singularities}. 
Bull. of the Amer. Math. Soc. 77 (1971), no. 4, 481--491.
\end{thebibliography}
\end{document}